\documentclass{amsart}
\usepackage[utf8]{inputenc}
\usepackage{enumerate}
\usepackage{setspace}
\usepackage{geometry}
\usepackage{graphicx}
\usepackage{subfigure}
\usepackage{amsthm,amsmath,amssymb}
\usepackage{caption}
\usepackage{graphicx, subfig}
\usepackage{float}
\usepackage{booktabs}
\usepackage[ruled]{algorithm2e}

\usepackage{fancyhdr}
\usepackage{mathrsfs}

\usepackage{framed}

\usepackage{caption}
\usepackage{subfigure}
\usepackage{indentfirst}
\newtheorem{theorem}{Theorem}[section]
\newtheorem{definition}{Definition}[section]

\newtheorem{remark}{Remark}[section]
\numberwithin{equation}{section}
\begin{document}

\title[GENERALIZED NEWTON-BUSEMANN LAW]{GENERALIZED NEWTON-BUSEMANN LAW FOR TWO-DIMENSIONAL STEADY HYPERSONIC-LIMIT EULER FLOWS PASSING RAMPS WITH SKIN-FRICTIONS}

\author{Aifang Qu}
\author{Xueying Su}
\author{Hairong Yuan}

\address[A. Qu]{Department of Mathematics, Shanghai Normal University,
Shanghai,  200234,  China}\email{\tt afqu@shnu.edu.cn}

\address[X. Su]{Center for Partial Differential Equations, School of Mathematical Sciences,
East China Normal University, Shanghai 200241, China}\email{\tt suxueying789@163.com}

\address[H. Yuan]{School of Mathematical Sciences,  Key Laboratory of Mathematics and Engineering Applications (Ministry of Education) \& Shanghai Key Laboratory of PMMP,  East China Normal University, Shanghai 200241, China}\email{\tt hryuan@math.ecnu.edu.cn}


%msc2020

\date{\today}

\maketitle
\allowbreak
\allowdisplaybreaks
\tableofcontents %disable for short paper
\begin{abstract}
    By considering Radon measure solutions for boundary value problems of stationary non-isentropic compressible  Euler equations on hypersonic-limit flows passing ramps with frictions on their boundaries, we construct  solutions with density containing Dirac measures supported on the boundaries of the ramps, which represent the infinite-thin shock layers under different assumptions on the skin-frictions. We thus derive corresponding generalizations of the celebrated Newton-Busemann law in hypersonic aerodynamics for distributions  of drags/lifts on ramps.
\end{abstract}

\section{Introduction}\label{sec1}

The problem of supersonic flow passing bodies is of fundamental importance in aeronautical engineering, as the distribution of aerodynamical drags/lifts provided by analysis of such problems provides a start for designation of supersonic vehicles. It is well-known that shock waves emerge (cf. Courant and Friedrichs \cite[Section 117]{12}) and the flow field around the body is generally described by discontinuous functions, and one needs the concept of integrable weak solutions, which are Lebesgue measurable functions, to study such problems. Moreover, it is observed that as Mach number of the upcoming supersonic flow goes to infinity, shock layer, i.e., the region bounded by the upwind boundary surface of the body and the shock front lying ahead of it, will be extremely thin and the density of mass will tend to infinity. In other words, mass will concentrate on the upwind boundary of the body. For more details on the physical aspects of inviscid hypersonic flows, we refer to Anderson \cite[Chapter 1]{new}, as well as the monograph \cite{HP} of Hayes and Probstein. Thus, for compressible Euler equations of inviscid flows, classical integrable weak solutions cannot depict the hypersonic-limit flow field correctly and one needs to introduce the notion of some kind of measure-valued solutions (Radon measure solutions) to characterize the infinite-thin shock layers.

 %obtain the stability of multidimensional delta-shock solution, which is actually a Radon measure, to the pressureless Riemann problem by the method of vanishing viscosity.  demonstrates the uniqueness of the weak solutions to the Cauchy initial problem of pressureless Euler system when the initial data is a Radon measure, rather than a measurable function. Pressureless Euler system are also used to model sticky particles and attain Radon measure solutions, see Grenier and Brenier \cite{GB}.

However, to our knowledge, for compressible Euler equations on physical problems with solid boundaries, no rigorous definitions of Radon measure solutions were raised until the works \cite{GQY,JQY,QY,QY2,QYZ2,QYZ}. Qu, Yuan, and Zhao \cite{QYZ} studied supersonic flow past a two-dimensional straight ramp and its hypersonic limit. They proposed a definition of Radon measure solution to the problem, and proved that the sequence of weak solutions containing  shock waves converges vaguely (as suitable measures when the Mach number of upcoming supersonic flow goes to infinity) to a singular measure solution of the hypersonic-limit flow, while the singular measure solution was calculated explicitly from the definition. The authors also derived the Newton's sine-squared law from this definition. Jin, Qu, and Yuan \cite{JQY} continued to investigate hypersonic-limit flow passing a curved ramp. They similarly introduced a definition of Radon measure solution and calculated the corresponding solution with density containing a Dirac measure, and corroborated the classical Newton-Busemann pressure law. For other applications of Radon measure solutions in solving high-Mach-number-limit of piston problems for polytropic gases and Chaplygin gas, as well as the problem of conical flows, see \cite{GQY,qsy,QY,QY2,QYZ2}. These results illustrated the mathematical rationality of the concept of Radon measure solutions. With help of such mathematical concepts, the complicated derivations of the Newton-Busemann law (cf. \cite[pp.67-71]{new}) were reduced to straightforward calculations, and one might obtain more general form of this law,  for example, for cones with attack angles \cite{QY, qsy}, which was not reported in the literature before.
%Lukas, Michael, Sahoo, and Abhrojyoti \cite{2} construct a measure-valued solution to an initial-boundary value problem on one-dimensional pressureless gas dynamics by using the method of generalized potentials and characteristic triangles.

All these works, including the classical Newton-Busemann law, assumed that the boundary surface of the obstacle is smooth, namely there is no any skin-friction. % Based on physical facts, there often exists damping resistance on the obstacles.
In this work, we  consider the problem of hypersonic-limit flow passing a curved ramp with skin-frictions, by proposing three different types of hypothetical conditions on such surface frictions.
For the first, we consider the case when the friction is determined by the velocity and the line density of gas in the infinite-thin shock layer. For the second, we suppose the skin friction is proportional to the pressure on the ramp. For the third, we assume the gas particles are all stuck to the ramp once they approach it and they no longer move, i.e., the infinite-thin shock layer is `frozen'. For each type of frictions, we solve the associated Radon measure solutions and derive the corresponding Newton-Busemann laws on distributions of lifts/drags. It turns out that once we have a proper definition, these achievements depend on finding solutions to some ordinary differential equations.  We also remark that one usually predicts the surface skin-friction by manipulating the viscous boundary layers under certain assumptions, see, for example, \cite[Section 6.9]{new}. We expect the ideas and results presented in this work provide new insights for such problems.

%Finding Radon measure solutions to these three problems turns out to be finding solutions to some ordinary differential equations. We also generalize the classical Newton-Busemann law to these three cases.

It should be noted that the hypersonic-limit flow is the pressureless Euler flow \cite{QYZ}. Studies on measure solutions of Cauchy problems for the pressureless Euler system  are enlightening and abundant. We just review a few works. E, Rykov, and Sinai \cite{ERS} presented an explicit construction of weak solutions, which are  Radon measures on the real line,  by means of a generalized variational principle. Bouchut \cite{BF} constructed  measure solutions which contains a delta shock to Riemann problems. For the uniqueness of measure solutions to Cauchy problems, see Li, Yang \cite{LY} and Huang, Wang \cite{HW}. See also \cite{GB, CL,SZ,YZ} for more results on delta-shock solutions of pressureless Euler systems.

For mathematical results on integrable weak solutions to supersonic flow passing bodies, we refer to Chen \cite{chen2022} for a review, and Kuang, Xiang, Zhang \cite{kxz} for verification of hypersonic similarity law. No concentration, thus no Newton-Busemann law,  was considered in these works.

In the following Section \ref{sec2}, we formulate, in an informal way, the three boundary value problems on hypersonic limit flow passing ramps with three types of skin-frictions. In Section \ref{sec3}, we  propose definitions of Radon measure solutions for these problems, and solve them explicitly. The main results are summarized as in Theorems \ref{thm1}, \ref{thm2},  \ref{thm3}, and Remark \ref{rem36}.

\section{Formulation of problems and skin-frictions}\label{sec2}

We now formulate the boundary value problems on hypersonic-limit Euler flows passing ramps.  Suppose a curved ramp lying in the right-half $xOy$-plane is given by
\begin{equation}
    \text{Ramp}~\dot{=}~\{(x,~y)~|~x\geq0,~0\le y\leq \Tilde{b}(x)\},~
\end{equation}
in which $\Tilde{b}(x)$ is a smooth function with $\Tilde{b}(0)=0$ and $\Tilde{b}'(x)\geq0$,  see Figure \ref{fig1}. The surface of the ramp is denoted as
\begin{equation}\label{eq22}
    R~\dot{=}~\{(x,~y)~|~x\geq0,~y= \Tilde{b}(x)\}.
\end{equation}
Uniform supersonic gas from the left-upper half-space flows towards the ramp with velocity parallel to the $x$-axis, and the space above the ramp which is filled with gas is denoted by
\begin{equation}
    \Omega~\dot{=}~\{(x,~y)~|~x\geq0,~y\geq \Tilde{b}(x)\}.
\end{equation}
The supersonic flow is assumed to be governed by the steady non-isentropic compressible Euler equations
%-----------------------------------fig1-----------
\begin{figure}[htb]
\centering
\includegraphics[scale=0.39]{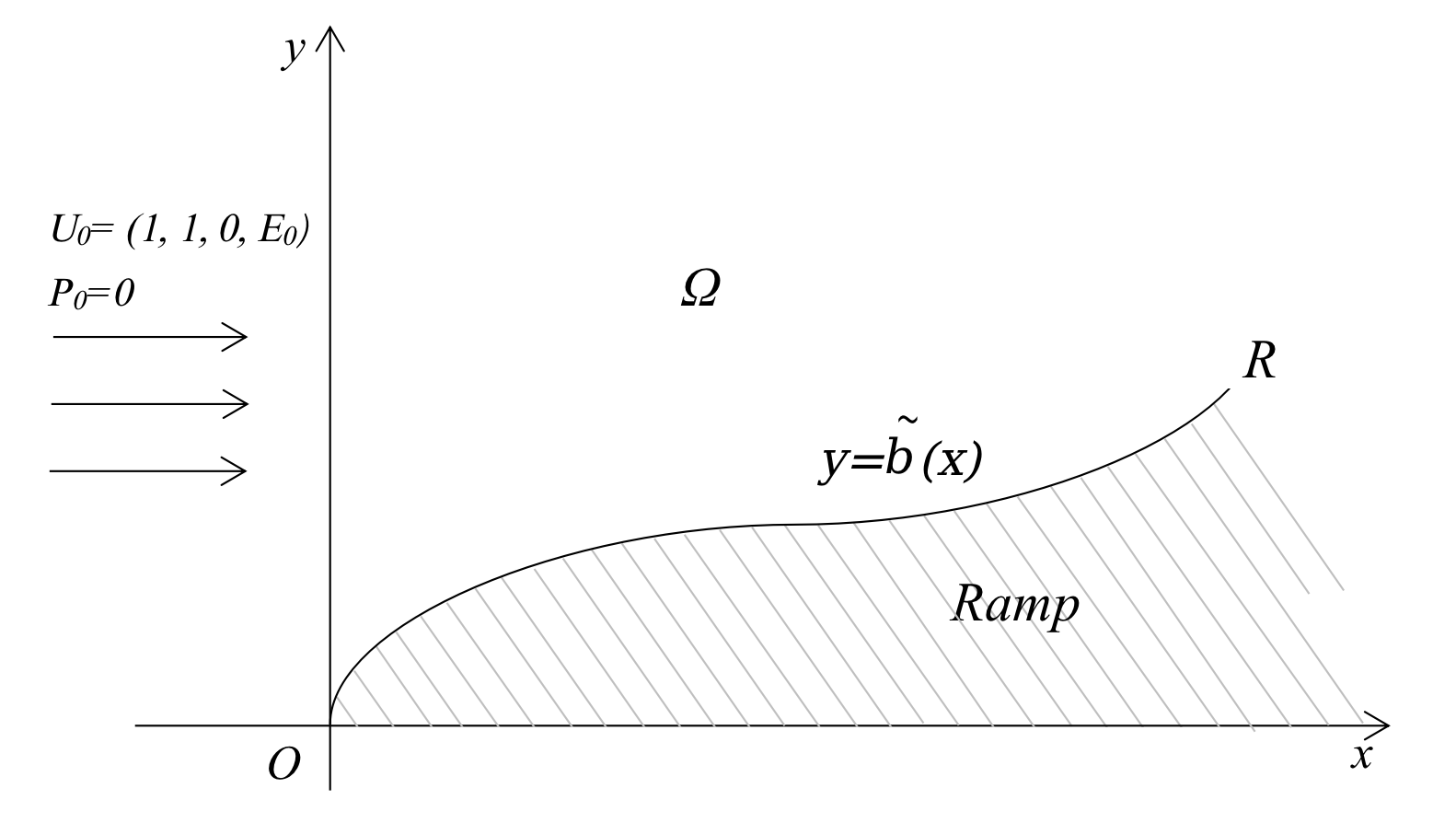}
\caption{Uniform hypersonic-limit flow passing a ramp.}\label{fig1}
\end{figure}
\begin{align}
    & (\rho u)_x+(\rho v)_y=0, \label{1.4} \\
    & (\rho u^2+p)_x+(\rho uv)_y=0,\label{1.5} \\
    & (\rho uv)_x+(\rho v^2+p)_y=0,\label{1.6} \\
    & (\rho u E)_x+(\rho vE)_y=0,\label{1.7}
\end{align}
which represent the conservation of mass, momentum and energy, respectively. In the system, $\rho$ stands for the mass density of gas, $u, v$  the argument of the velocity $V$ along $x$ and $y$-axis respectively, and $E$ the total enthalpy per unit mass of gas. The pressure of the gas is given by
\begin{equation}
    p=\frac{\epsilon}{\epsilon+1}\rho\cdot(E-\frac{1}{2}(u^2+v^2)). \label{1.8}
\end{equation}
Here, $\gamma\doteq\epsilon+1$ (with $\epsilon>0$) is the adiabatic exponent of polytropic gas. As in \cite{JQY}, by suitable scaling, we may assume that the upstream flow is
\begin{equation}
    U_0\doteq(\rho_0,~ V_0=(u_0,~ v_0),~E_0)=(1,~1,~0,~E_0),\quad E_0>\frac{1}{2}, \label{1.9}
\end{equation}
hence
\begin{equation}
    p_0=\frac{\epsilon}{\epsilon+1}(E_0-\frac{1}{2}). \label{1.10}
\end{equation}
On $R$,~i.e., the surface of the ramp, the slip condition
\begin{equation}
    V\cdot\mathbf{n}=0 \label{1.11}
\end{equation}
is subjected, where $\mathbf{n}$ denotes the unit outward normal vector of $\Omega$ on $R$. The equations \eqref{1.4}-\eqref{1.11} consist the classical formulation to the problem of supersonic flow passing a ramp.

As for limiting hypersonic flows, i.e., the Mach number of the upcoming flow, which we denote as $M_0$, tends to infinity, combining \eqref{1.9} and the fact that $c\doteq\sqrt{(\epsilon+1)p/\rho}$, where $c$ represents the local sound speed, one has
\begin{equation}
    M_0^2=\frac{V_0^2}{c_0^2}=\frac{1}{\epsilon (E_0-\frac{1}{2})}.
\end{equation}
It shows that $M_0$ increasing to $\infty$ equals to $\epsilon$ decreasing to $0$, and therefore, by \eqref{1.10}, the pressure of the upcoming gas is $p_0=0$. Hence, hypersonic-limit flow may be considered as pressureless flow if no physical boundary occurs, cf. Remark \ref{rm32}.

%~The above boundary value problem with $\epsilon=0$ models the hypersonic-limit flow passing a ramp without damping.~In this work, we further study three problems (Problem 1, Problem 2 and Problem 3) below if the ramp has damping force to the gas on its boundary, and then generalize the classical Newton-Busemann law.
We now specify the skin-frictions. %Apart from the slip boundary condition, damping boundary conditions will also be proposed.
In the sequel, we use the arc-length parameter $s$ of the graph of $y=\Tilde{b}(x)$, with $\mathrm{d}s=\sqrt{1+\Tilde{b}'(x)^2}\,\mathrm{d}x,$ which means $$s(x)=\int_0^x\sqrt{1+\Tilde{b}'(t)^2}\,\mathrm{d}t.$$
By the inverse function theorem,  there exists a (at least) $C^2$ function $\psi$, such that $x=\psi(s)$. We write $b(s)\doteq\Tilde{b}(\psi(s))$, and the coordinates on $R$ becomes $(\psi(s),b(s))$. We set \begin{align}
w(s)\doteq\sqrt{u(\psi(s),b(s))^2+v(\psi(s),b(s))^2}\label{eq213new}
\end{align}
to be the (scalar) speed of gas particles along the ramp. We use $\dot{(\cdot)}$ to denote the derivative of functions with respect to $s.$~The unit outward normal vector of $\Omega$ on $R$ is $\mathbf{n}(s)=(\dot{b}(s),-\dot{\psi}(s)) $ and the unit tangent vector along $R$ is $\mathbf{t}(s)=(\dot{\psi}(s),\dot{b}(s))$. Furthermore, if we denote $f_1(s),~f_2(s)$ as the force acting on the gas  by the ramp along $x$ and $y$-axis, then the tangential skin friction $f$ and normal pressure $N$ received by the ramp are
\begin{align}
    &f(\psi(s),b(s))\doteq-(f_1,f_2)\cdot\mathbf{t}=-f_1\dot{\psi}(s)-f_2\dot{b}(s), \label{e10}\\
    &N(\psi(s),b(s))\doteq-(f_1,f_2)\cdot \mathbf{n}=-f_1\dot{b}(s)+f_2\dot{\psi}(s).\label{e1}
\end{align}
We assume that $N\ge0$ in this work, namely particles always  impinge on the ramp, cf. Remark \ref{rm34}. Notice that $f$ should also be nonnegative, because the particles should move along the direction of $\mathbf{t}$.
%Using these notions, we formulate our problems with additional boundary conditions on dampings.

For the first case, we assume that the skin-friction acting on the particles by the ramp's boundary is determined by $$f(\psi(s),b(s))=kw_\rho(s) (w(s))^\alpha, \quad \alpha\geq1.$$ Here $k>0$ is a given constant and $w_\rho(s)$ is the line density of gas mass on $R$ if there is concentration of mass, as to be clarified in \eqref{2.16}. This is motivated by the standard resistance formula for $\alpha=2$, and linear damping for $\alpha=1$. In this situation, by \eqref{e10}, there holds
\begin{equation}
f_1\dot{\psi}(s)+f_2\dot{b}(s)=-kw_\rho(s) (w(s))^\alpha. \label{1.12}
\end{equation}
%Notice that the minus sign means the force is pointing to the direction that $s$ decreasing.
Then informally the first problem we consider is the following:

\textsc{Problem 1:} Find a solution to \eqref{1.4}-\eqref{1.7}, \eqref{1.9}, \eqref{1.10}($\epsilon=0$), \eqref{1.11} and \eqref{1.12}.

\medskip

For the second case, we consider the usual assumption that the friction is proportional to the stress, namely given by $f(\psi(s),b(s))=\eta N(\psi(s),b(s))$, where $\eta>0$ is the dynamic friction coefficient between the ramp and the flow in the infinite-thin shock layer. Again, by \eqref{e10} and \eqref{e1}, we derive the boundary condition
\begin{equation}
f_1\dot{\psi}(s)+f_2\dot{b}(s)=\eta\cdot(f_1\dot{b}(s)-f_2\dot{\psi}(s)).\label{1.13}
\end{equation}
We thus have informally the second problem:

\textsc{Problem 2:} Find a solution to \eqref{1.4}-\eqref{1.7}, \eqref{1.9}, \eqref{1.10}($\epsilon=0$), \eqref{1.11} and \eqref{1.13}.

\medskip

We are also interested in the situation that particles of the upcoming flow immediately stick to the ramp once they reach it and do not move any more. To solve this, we apply an approximation as follows. For $k=0$ in Problem 1 (or $\eta=0$ in Problem 2), which represents the case
of null skin-friction, one calculates the corresponding velocity along the ramp, denoted as $(u_1(\psi(s),b(s)), v_1(\psi(s),b(s)))$. Based on this, we introduce a coefficient $\mu,~0<\mu<1,$
and require that
\begin{equation}
    u(\psi(s),b(s))=\mu u_1,\quad v(\psi(s),b(s))=\mu v_1. \label{1.14}
\end{equation}
Letting $\mu\to0$, we will get a formal answer. We thus consider the following:

\textsc{Problem 3:} Find a solution to \eqref{1.4}-\eqref{1.7}, \eqref{1.9}, \eqref{1.10}($\epsilon=0$), \eqref{1.11} and \eqref{1.14}.

%We will give a specific definition of Problem 3 hereinafter (see Problem $3'$).

\section{Measure solutions and main results}\label{sec3}

To carry out mathematical analysis of the aforementioned three problems, we need their rigorous formulations, namely a concept of measure solutions, which should include the infinite-thin shock layers observed by physicists as indicated in Section \ref{sec1}. It turns out that the theory of Radon measures provides a suitable framework, since these measures could be considered as linear continuous functionals on the space of compactly supported continuous functions, thus easy to take them as solutions to the compressible Euler equations by extending the well-accepted definitions of integrable weak solutions. Since the infinite-thin shock layer attached to a ramp is the curve $R$ (see \eqref{eq22}), we need Dirac measures supported on the curve to represent it. The unknowns are reduced to the weights of the Dirac measures, which are functions on $R$. Mathematically, by this reduction of order, we reduce a problem on partial differential equations to problems of ordinary differential equations, which could be solved explicitly for many cases, thus demonstrate efficiency of this approach.

\subsection{Radon measure solutions}\label{sec31}

Let $\mathcal{B}$ be the Borel $\sigma$-algebra on $\mathbb{R}^2$ and $m$ a Radon measure on $(\mathbb{R}^2,\mathcal{B})$. The pairing between $m$ and a test function $\Tilde{\phi}(x,y)\in C_0(\mathbb{R}^2)$, where $C_0(\mathbb{R}^2)$ is the set of continuous functions on $\mathbb{R}^2$ with compact supports, is given by
$$\langle m,~\Tilde{\phi}\rangle=\int_{\mathbb{R}^2}\Tilde{\phi}(x,y)\,\mathrm{d}m.$$
For example, $W(s)\delta_R$, the Dirac measure supported on $R$ with weight $W(s)$, is defined by
$$\langle W(s)\delta_R,~\Tilde{\phi}\rangle=\int_{R}W(s)~\Tilde{\phi}(\psi(s),b(s))\,\mathrm{d}s.$$
We also denote $\mathcal{L}^2$ as the Lebesgue measure on $\mathbb{R}^2$, and write $\mu\ll\varrho$ if the measure $\mu$ is absolutely continuous with respect to a nonnegative measure $\varrho$, with Radon-Nikodym derivative $\mathrm{d}\mu/\mathrm{d}\varrho$.  We now raise a definition of Radon measure solutions to the Problems 1, 2 and 3. The motivation of such a definition could also be found in \cite{JQY,QY,QYZ2,QYZ}.

\begin{definition}
    For given $\epsilon\geq0$,~let $m^0,~ n^0 ,~m^1,~ n^1,~ m^2,~ n^2,~m^3,~n^3,~\varrho,~\wp$ be Radon measures on $\overline{\Omega}$, and $w_{f_1}(s),~w_{f_2}(s)$ be locally integrable functions on $[0,\infty)$. Suppose that

(\romannumeral1) for any  $ \Tilde{\phi}\in C_0^1(\mathbb{R}^2)$(continuously differentiable functions with compact supports), there hold
\begin{align}
    & \langle m^0,~\Tilde{\phi}_x\rangle+\langle n^0,~\Tilde{\phi}_y\rangle+\int_0^\infty\rho(0,~y)u(0,~y)\Tilde{\phi}(0,y)\,\mathrm{d}y=0, \label{2.1} \\
    & \langle m^1,~\Tilde{\phi}_x\rangle+\langle n^1,~\Tilde{\phi}_y\rangle+\langle\wp,~\Tilde{\phi}_x\rangle+\langle w_{f_1}\delta_R,~\Tilde{\phi}\rangle+\int_0^\infty(\rho_0u_0^2+p_0)\Tilde{\phi}(0,y)\,\mathrm{d}y=0, \label{2.2} \\
    & \langle m^2,~\Tilde{\phi}_x\rangle+\langle n^2,~\Tilde{\phi}_y\rangle+\langle\wp,~\Tilde{\phi}_y\rangle+\langle w_{f_2}\delta_R,~\Tilde{\phi}\rangle+\int_0^\infty(\rho_0u_0v_0)\Tilde{\phi}(0,y)\,\mathrm{d}y=0,\label{2.3} \\
    & \langle m^3,~\Tilde{\phi}_x\rangle+\langle n^3,~\Tilde{\phi}_y\rangle+\int_0^\infty(\rho_0u_0E_0)\Tilde{\phi}(0,y)\,\mathrm{d}y=0. \label{2.4}
\end{align}

(\romannumeral2) $\varrho,  \wp$ are non-negative measures, $m^i, n^i\ (i=1, 2, 3)$ and $\wp$ are absolutely continuous with respect to $\varrho$, and there also exist  $\varrho$-a.e. functions $u, v, E$, such that the Radon-Nikodym derivatives satisfy
\begin{align}
    & u=\frac{\mathrm{d}m^0}{\mathrm{d}\varrho}=\frac{\mathrm{d}m^1/\mathrm{d}\varrho}
    {\mathrm{d}m^0/\mathrm{d}\varrho}=\frac{\mathrm{d}n^1/\mathrm{d}\varrho}{\mathrm{d}n^0/\mathrm{d}\varrho},\label{2.5} \\
    & v=\frac{\mathrm{d}n^0}{\mathrm{d}\varrho}=\frac{\mathrm{d}m^2/\mathrm{d}\varrho}{\mathrm{d}m^0/\mathrm{d}\varrho}
    =\frac{\mathrm{d}n^2/\mathrm{d}\varrho}{\mathrm{d}n^0/\mathrm{d}\varrho},\label{2.6}\\
    & E=\frac{\mathrm{d}m^3/\mathrm{d}\varrho}{\mathrm{d}m^0/\mathrm{d}\varrho}=
    \frac{\mathrm{d}n^3/\mathrm{d}\varrho}{\mathrm{d}n^0/\mathrm{d}\varrho}.\label{2.7}
\end{align}
In addition, $\varrho$ is absolutely continuous with respect to the Dirac measure $\delta_R$, with the Radon-Nikodym derivative $w_\rho(s)\doteq\mathrm{d}\varrho/\mathrm{d}\delta_R$ satisfying, respectively,
for Problem 1,
\begin{equation}
    w_{f_1}\dot{\psi}(s)+w_{f_2}\dot{b}(s)=-kw_\rho w^\alpha; \label{2.8}
\end{equation}
for Problem 2,
\begin{equation}
    w_{f_1}\dot{\psi}(s)+w_{f_2}\dot{b}(s)=\eta\cdot\Big(w_{f_1}\dot{b}(s)-w_{f_2}\dot{\psi}(s)\Big); \label{2.9}
\end{equation}
for Problem 3,
\begin{equation}
    u(\psi(s),b(s))=\mu u_1,\quad v(\psi(s),b(s))=\mu v_1. \label{2.10}
\end{equation}

(\romannumeral3) if $\varrho, \wp\ll\mathcal{L}^2$, and their Radon-Nikodym derivations are
\begin{equation}
    \rho=\frac{\mathrm{d}\varrho}{\mathrm{d}\mathcal{L}^2},\quad p=\frac{\mathrm{d}\wp}{\mathrm{d}\mathcal{L}^2},
\end{equation}
then \eqref{1.8} holds ($\mathcal{L}^2$-a.e.) and classical entropy condition is valid for discontinuities of functions $\rho, u, v, E$ in this case.

Then we call $(\varrho, u, v, E, w_{f_1}, w_{f_2})$ a Radon measure solution to Problem 1, Problem 2 and Problem 3, respectively.
\end{definition}

\begin{remark}{\em
The speed $w$ in \eqref{2.8} is given by \eqref{eq213new}, where $u(\psi(s), b(s)), v(\psi(s), b(s))$ are well-defined, thanks to \eqref{2.5}, \eqref{2.6}, and the requirement $\varrho\ll\delta_R$.
Also note that $u_1$ and $v_1$ in \eqref{2.10} are shown in \eqref{2.47} as $u, v$ there, calculated in an independent way.
    }
\end{remark}

\begin{remark}\label{rm32}{\em
    The conservation laws \eqref{2.1}-\eqref{2.4} are derived from equations \eqref{1.4}-\eqref{1.7} via integration-by-parts and replacing integrals by pairings. The measure $\wp$ in the term $\langle\wp,~\Tilde{\phi}_x\rangle$ in \eqref{2.2} (or $\langle\wp,~\Tilde{\phi}_y\rangle$ in \eqref{2.3}) stands for the static pressure, and the term  $\langle w_{f_1}\delta_R,~\Tilde{\phi}\rangle$ in \eqref{2.2} (or $\langle w_{f_2}\delta_R,~\Tilde{\phi}\rangle$ in \eqref{2.3}) stands for the  dynamic pressure. For hypersonic-limit flow, the static pressure vanishes, while the dynamic pressure always presents if there is a solid physical boundary $R$. We see an advantage of Radon measure solution is that it decomposes these two parts of pressure clearly.
}
    %In the hypersonic-limit problem, say $\epsilon=0$, the pressure measure $\wp$ actually equals to $0$.
\end{remark}
\begin{remark}{\em Notice that $(w_{f_1}, w_{f_2})$ means exactly the force $(f_1, f_2)$ on the gas by the ramp. For a motivation of this understanding, see Remark 1.5 in \cite{JQY}. (Be careful that  $\mathbf{n}$ is the inner normal vector in \cite{JQY}).  We emphasize that $w_{f_1}(s),~w_{f_2}(s)$ are now independent undetermined functions because of the additional boundary conditions on skin frictions. This is a major difference from Definition 1.1 in \cite{JQY}.
}
\end{remark}

\subsection{Infinite-thin shock layers} \label{sec32}
Let $\mathbb{I}_\Omega$ be the indicator function of a set $\Omega$, which means $\mathbb{I}_\Omega(x,y)=1$ if $(x,y)\in\Omega$, and $\mathbb{I}_\Omega(x,y)=0$ otherwise. It is natural to employ Dirac measures supported on the boundary $R$ to represent the infinite-thin shock layers in hypersonic-limit flow passing the ramp. Thus, we suppose a Radon measure solution of Problem 1 (Problem 2 and 3) determined by the following regular-singular decompositions:
\begin{align}
    & m^0=\rho_0u_0\mathcal{L}^2\mathbb{I}_\Omega+w_m^0(s)\delta_R=\mathcal{L}^2\mathbb{I}_\Omega+w_m^0(s)\delta_R,~n^0=\rho_0v_0\mathcal{L}^2\mathbb{I}_\Omega+w_n^0(s)\delta_R=w_n^0(s)\delta_R, \label{2.12} \\
    & m^1=\rho_0u_0^2\mathcal{L}^2\mathbb{I}_\Omega+w_m^1(s)\delta_R=\mathcal{L}^2\mathbb{I}_\Omega+w_m^1(s)\delta_R,\quad n^1=w_n^1(s)\delta_R,\label{2.13} \\
    & m^2=\rho_0u_0v_0\mathcal{L}^2\mathbb{I}_\Omega+w_m^2(s)\delta_R=w_m^2(s)\delta_R,\quad n^2=w_n^2(s)\delta_R,\label{2.14} \\
    & m^3=\rho_0u_0E_0\mathcal{L}^2\mathbb{I}_\Omega+w_m^3(s)\delta_R= w_m^3(s)\delta_R,\quad n^3=w_n^3(s)\delta_R,\label{2.15}\\
    & \varrho=\mathcal{L}^2\mathbb{I}_\Omega+w_\rho(s)\delta_R,\quad \wp=0.\label{2.16}
\end{align}
Here, $w_m^i,~w_n^i, ~(i=0,1,2,3)$ and $w_\rho(s)$ are all functions to be determined on the solid boundary $R$. And $w_\rho(s)$ represents the weight of mass, namely, the line density of gas on $R$ if there is mass concentration. We suppose $\wp=0$ in \eqref{2.16}, because if we assume that $\wp=0+w_p\delta_R$, where $w_p$ is the weight of static pressure on the surface of the ramp, then by calculating as in the following procedure, we also have $w_p=0$. Thereinafter we set $\phi(s,y)\doteq\Tilde{\phi}(\psi(s),y).$

We now derive the (ordinary) differential equations obeyed by these weights.
%Substituting \eqref{2.12}-\eqref{2.16} into  \eqref{2.1}-\eqref{2.4} respectively, and by \eqref{2.5}-\eqref{2.10}, it turns out that constructing the Radon measure solution of Problem 1 (2, 3) is reduced to finding solutions to some ordinary differential equations with respect to $s$. Explicitly,
Substituting \eqref{2.12} into \eqref{2.1}, one has
\begin{multline*}
 \int_\Omega\Tilde{\phi}_x\,\mathrm{d}x\mathrm{d}y+\int_0^\infty w_m^0(s)\Tilde{\phi}_x(\psi(s),~b(s)) \,\mathrm{d}s+\int_0^\infty w_n^0(s)\Tilde{\phi}_y(\psi(s),b(s)) \,\mathrm{d}s\\
 +\int_0^\infty \Tilde{\phi}(0,y)\,\mathrm{d}y=0,
\end{multline*}
and it follows
\begin{align*}
   - & \int_0^\infty\Tilde{\phi}(0,~y)\dot{b}(s)\,\mathrm{d}s+\int_0^\infty \phi(s,b(s))\dot{b}(s)\mathrm{d}s+\int_0^\infty w_m^0(s) \phi_s(s,~b(s))(\dot{\psi}(s))^{-1}\,\mathrm{d}s\\
    & ~~~~~~~~~~~~~~~~~~~~~~~~~~~~~~~~~~~~~~~~~~~~~~~~~~ \quad\qquad+\int_0^\infty w_n^0(s)\phi_y(s,~b(s)) \,\mathrm{d}s+\int_0^\infty \Tilde{\phi}(0,~y)\,\mathrm{d}y=0.
\end{align*}
Noticing that
\begin{align*}
    &\int_0^\infty w_m^0(s) \phi_s(s,~b(s))(\dot{\psi}(s))^{-1}\,\mathrm{d}s\\ \nonumber
    =& -w_m^0(0)\dot{\psi}^{-1}(0)\phi(0,~0)-\int_0^\infty\frac{\mathrm{d}\Big(w_m^0(s)(\dot{\psi}(s))^{-1}\Big)}
    {\mathrm{d}s}\phi(s,b(s))\,\mathrm{d}s\\ &~~~~~~~~~~~~~~~~~~~~~~~~~~~~~~~~~~~~~~~~~~~~~~~~~~~~~~~~-\int_0^\infty w_m^0(s)(\dot{\psi}(s))^{-1}\dot{b}(s)\phi_y(s,~b(s))\,\mathrm{d}s,
\end{align*}
we acquire
\begin{equation*}
    \begin{aligned}
w_m^0(0)\phi(0,~0)(\dot{\psi}(0))^{-1}+\int_0^\infty & (\frac{\mathrm{d}\Big(w_m^0(s)(\dot{\psi}(s))^{-1}\Big)}{\mathrm{d}s}-\dot{b}(s))\phi(s,~b(s))\,\mathrm{d}s\\
    &+\int_0^\infty\Big(\dot{b}(s)w_m^0(s)(\dot{\psi}(s))^{-1}-w_n^0(s) \Big)\phi_y(s,~b(s))\,\mathrm{d}s=0.
    \end{aligned}
\end{equation*}
By the arbitrariness of $\phi$, one has
\begin{align}
    & w_m^0(0)=0,\quad w_m^0(s)=b(s)\dot{\psi}(s),\\
    & w_n^0(s)=b(s)\dot{b}(s).
\end{align}
Similarly one gets
\begin{align}
    & w_m^3(0)=0,\quad w_m^3(s)=E_0b(s)\dot{\psi}(s), \\
    & w_n^3(s)=E_0b(s)\dot{b}(s).
\end{align}
Substituting \eqref{2.13}, \eqref{2.16} into \eqref{2.2}, and \eqref{2.14}, \eqref{2.16} into \eqref{2.3}, one also obtains that
\begin{align}
    & w_m^1(0)=0,\quad w_m^1(s)=\dot{\psi}(s)\big(b(s)+\int_0^s w_{f_1}(t)\,\mathrm{d}t \big),\\
    & w_n^1(s)=\dot{b}(s)\big(b(s)+\int_0^s w_{f_1}(t)\,\mathrm{d}t \big),
\end{align}
as well as
\begin{align}
    & w_m^2(0)=0,\quad w_m^2(s)=\dot{\psi}(s)\int_0^s w_{f_2}(t)\,\mathrm{d}t,\\
    & w_n^2(s)=\dot{b}(s)\int_0^s w_{f_2}(t)\,\mathrm{d}t.
\end{align}
Then from \eqref{2.5}-\eqref{2.7}, we infer that
\begin{align}
    & u(\psi(s),b(s))=\displaystyle{\frac{b(s)+\int_0^s w_{f_1}(t)\,\mathrm{d}t}{b(s)}}, \quad v(\psi(s),b(s))=\displaystyle{\frac{\int_0^s  w_{f_2}(t)\,\mathrm{d}t}{b(s)}}, \label{2.25}\\
    & \varrho(\psi(s),b(s))=\displaystyle{\frac{b(s)^2\dot{b}(s)}{\int_0^s w_{f_2}(t)\,\mathrm{d}t}}\delta_R,\quad E(\psi(s),b(s))=E_0. \label{2.26}
\end{align}
The above equations indicate that
    \begin{align}
     & \frac{b(s)+\int_0^s w_{f_1}(t)\mathrm{d}t}{b(s)}  =w(s)\dot{\psi}(s),  \label{2.27}\\
     & \frac{\int_0^s w_{f_2}(t)\mathrm{d}t}{b(s)}=w(s)\dot{b}(s),\label{2.28}\\
    &
    w_\rho(s)=\frac{w_n^0(s)}{v(s)}=\frac{b(s)}{w(s)}. \label{2.29}
\end{align}
Solving \eqref{2.27} and \eqref{2.28}, one gets
\begin{align}
    & w_{f_1}(s)=\dot{w}\dot{\psi}b(s)+w\ddot{\psi}b(s)+w\dot{\psi}\dot{b}-\dot{b}(s),\\
    & w_{f_2}(s)=\dot{w}\dot{b}b(s)+w\ddot{b}b(s)+w\dot{b}^2(s).
\end{align}
Therefore, thanks to $\dot{b}^2+\dot{\psi}^2\equiv1$ and $\dot{b}\ddot{b}+\dot{\psi}\ddot{\psi}\equiv0$,  we obtain
\begin{align}
&N(\psi(s),b(s))=-(w_{f_1},w_{f_2})\cdot\mathbf{n}={wb(\ddot{b}\dot{\psi}(s)-\dot{b}\ddot{\psi}(s))+\dot{b}^2(s),} \label{2.32}\\
&f(\psi(s),b(s))=-(w_{f_1},w_{f_2})\cdot\mathbf{t}={-\dot{w}b(s)-\dot{b}w(s)+\dot{b}\dot{\psi}(s).}\label{2.33}
\end{align}
\begin{remark}\label{rm34}{\em
    We require $N(\psi(s),b(s))>0$ to guarantee that the mass concentrates on the surface of the ramp. In fact, when $\dot{b}(0)>0,$ $N(\psi(s),b(s))>0$ if $s$ is within a small range from 0.
    }
\end{remark}

\subsection{Main results and their proofs}\label{sec33}

For \textsc{Problem 1}, \eqref{2.32}, \eqref{2.33} together with \eqref{2.8} yield that
\begin{framed}
    \begin{equation}
     \dot{w}b(s)+\dot{b}w(s)-\dot{b}\dot{\psi}(s)=-kw_\rho(s) w^\alpha(s)=-kb(s)w^{\alpha-1}(s). \label{2.34}
\end{equation}
\end{framed}

For $\alpha=1$, the above equation becomes
\begin{equation*}
    (wb)^\cdot(s)=\dot{b}\dot{\psi}(s)-kb(s),\quad wb(0)=0,
\end{equation*}
and one solves it to have
\begin{equation}
w(s)=\frac{1}{b(s)}\int_0^s(\dot{b}\dot{\psi}(t)-kb(t))\,\mathrm{d}t. \label{2.35}
\end{equation}
The definition \eqref{eq213new} of $w$ demands that $w(s)\geq0$. However, by \eqref{2.35}, $w(s)$ might be negative when $s$ is sufficiently large,\footnote{
For example, take $\tilde{b}(x)=x$, i.e., $b(s)=\frac{\sqrt{2}}{2}s$. Then $\dot{b}(s)=\frac{\sqrt{2}}{2}$, and $\dot{\psi}(s)=\frac{1}{\sqrt{2}}$. For $k=1$, one has $\dot{b}\dot{\psi}(t)-kb(t))=\frac{1}{2}-\frac{\sqrt{2}}{2}t<0$ for $t\in(\frac{\sqrt{2}}{2}, \infty)$.} which indicates that the case $\alpha=1$ might not be a reasonable model to describe the skin-friction.

For $\alpha=2$, one solves the corresponding equation
\begin{equation*}
    (wb)^\cdot(s)=-k(wb(s))+\dot{b}\dot{\psi}(s)
\end{equation*}
to have
\begin{equation}
    w(s)=\frac{1}{b(s)\mathrm{e}^{ks}}\int_0^s\dot{b}\dot{\psi}(t)\mathrm{e}^{kt}\,\mathrm{d}t.\label{2.36}
\end{equation}

When $\alpha$ takes other values, there still remains a problem to solve \eqref{2.34}.~One can use numeral methods to give approximate solutions.

\medskip
For \textsc{Problem 2}, \eqref{2.32}, \eqref{2.33} and the boundary condition \eqref{2.9} yield
\begin{framed}
\begin{equation}
    \dot{w}b(s)+\dot{b}w(s)-\dot{b}\dot{\psi}(s)=-\eta\cdot\Big(wb(\ddot{b}\dot{\psi}(s)-\dot{b}\ddot{\psi}(s))+\dot{b}^2(s)\Big).\label{2.37}
\end{equation}
\end{framed}
Simplifying the above equation, we have
\begin{equation*}
(wb)^\cdot(s)=\Big(-\eta\ddot{b}\dot{\psi}(s)+\eta\dot{b}\ddot{\psi}(s)\Big)wb(s)-\eta\dot{b}^2(s)+\dot{b}\dot{\psi}(s).
\end{equation*}
Set $$I(s)\doteq\exp(\int_0^s-\eta(\ddot{b}\dot{\psi}(t)-\dot{b}\ddot{\psi}(t))\,\mathrm{d}t).$$
Then the solution to \eqref{2.37} is
\begin{equation}
    w(s)=\frac{I(s)}{b(s)}\int_0^s\Big(-\eta\dot{b}^2(t)+\dot{b}\dot{\psi}(t)\Big)I(t)^{-1}\,\mathrm{d}t.
\end{equation}

%\begin{align}
 %   & \Tilde{I}(x)=I(s(x))=\exp(\int_0^x\frac{\eta\Tilde{b}''(t)}{1+\Tilde{b}'(t)^2}\mathrm{d}t),\\
    %& \Tilde{w}(x)=\frac{\Tilde{I}(x)}{\Tilde{b}(x)}\int_0^x\frac{\Tilde{b}'(t)(\eta\Tilde{b}'(t)+1)}{\Tilde{I}(t)\sqrt{1+\Tilde{b}'(t)^2}}\mathrm{d}t.
%\end{align}
Therefore, turning back to the $(x,y)$-coordinates, and also recalling $s(x)=\int_0^x\sqrt{1+\Tilde{b}'(t)^2}\,\mathrm{d}t$, we have the following theorems.

\begin{theorem}\label{thm1}
For $\alpha=2$, set $H_k(x)\doteq\int_0^x\frac{\Tilde{b}'(t)}{\sqrt{1+\Tilde{b}'(t)^2}}\mathrm{e}^{ks(t)}\,\mathrm{d}t$. Then Problem 1 admits a Radon measure solution given by
    \begin{align}
        & \varrho=\mathcal{L}^2\mathbb{I}_\Omega+\frac{\Tilde{b}(x)^2\mathrm{e}^{ks(x)}}{H_k(x)}~\delta_R,\quad E=E_0\mathbb{I}_\Omega+E_0\mathbb{I}_R,\\
         & u=\mathbb{I}_\Omega+\frac{H_k(x)}{\Tilde{b}(x)\sqrt{1+\Tilde{b}'(x)^2}\mathrm{e}^{ks(x)}}~\mathbb{I}_R,\quad v=\frac{\Tilde{b}'(x)H_k(x)}{\Tilde{b}(x)\sqrt{1+\Tilde{b}'(x)^2}\mathrm{e}^{ks(x)}}~\mathbb{I}_R.
    \end{align}
Moreover, the normal pressure $N$ and skin-friction $f$ on the ramp are
\begin{equation}
    N(x,\Tilde{b}(x))=\frac{\Tilde{b}''(t)H_k(x)}{(1+\Tilde{b}'(t)^2)^\frac{3}{2}\mathrm{e}^{ks(x)}}+\frac{\Tilde{b}'(t)^2}{1+\Tilde{b}'(t)^2},\quad f(x,\Tilde{b}(x))=\frac{k H_k(x)}{\mathrm{e}^{ks(x)}}. \label{2.41}
\end{equation}
\end{theorem}
\begin{proof}
Direct calculation with \eqref{2.36} shows that
\begin{align*}
     & u(\psi(s),b(s))=w(s)\dot{\psi}(s)=\frac{\dot{\psi}(s)}{b(s)\mathrm{e}^{ks}}\int_0^s\dot{b}\dot{\psi}(t)\mathrm{e}^{kt}\,\mathrm{d}t, \\
     & v(\psi(s),b(s))=w(s)\dot{b}(s)=\frac{\dot{b}(s)}{b(s)\mathrm{e}^{ks}}\int_0^s\dot{b}\dot{\psi}(t)\mathrm{e}^{kt}\,\mathrm{d}t,\\
    & w_\rho(s)=\frac{b(s)}{w(s)}=b^2(s)\mathrm{e}^{ks(x)}\Big(\int_0^s\dot{b}\dot{\psi}(t)\mathrm{e}^{kt}\,\mathrm{d}t\Big)^{-1}.
\end{align*}
Using the $x$ parameter, we infer
\begin{equation*}
    \Tilde{w}(x)=w(s(x))=\frac{1}{\Tilde{b}(x)\mathrm{e}^{ks(x)}}\int_0^x\frac{\Tilde{b}'(t)}{\sqrt{1+\Tilde{b}'(t)^2}}\mathrm{e}^{ks(t)}
    \,\mathrm{d}t=\frac{H_k(x)}{\Tilde{b}(x)\mathrm{e}^{ks(x)}}.
\end{equation*}
Hence,
\begin{align*}
    & u(x,\Tilde{b}(x))=\frac{H_k(x)}{\Tilde{b}(x)\sqrt{1+\Tilde{b}'(x)^2}\mathrm{e}^{ks(x)}},\quad  v(x,\Tilde{b}(x))=\frac{\Tilde{b}'(x)H_k(x)}{\Tilde{b}(x)\sqrt{1+\Tilde{b}'(x)^2}\mathrm{e}^{ks(x)}},\\
    & w_\rho(s(x))=\Tilde{b}(x)^2\mathrm{e}^{ks(x)}\Big(H_k(x)\Big)^{-1}.
\end{align*}
From \eqref{2.32} and \eqref{2.33}, we have
\begin{align*}
     & N(x,\Tilde{b}(x))=N(\psi(s),b(s))=\frac{\Tilde{b}''(t)}{(1+\Tilde{b}'(t)^2)^\frac{3}{2}\mathrm{e}^{ks(x)}}H_k(x)+\frac{\Tilde{b}'(t)^2}{1+\Tilde{b}'(t)^2},\\
    & f(x,\Tilde{b}(x))=f(\psi(s),b(s))=k\Tilde{b}\Tilde{w}(x)=\frac{k H_k(x)}{\mathrm{e}^{ks(x)}}.
\end{align*}
\end{proof}
\begin{theorem}\label{thm2}
    Problem 2 has a Radon measure solution given by
    \begin{align}
    & \varrho=\mathcal{L}^2\mathbb{I}_\Omega+\frac{\Tilde{b}^2(x)}{\Tilde{I}(x)W_\eta(x)}\delta_R,\quad E=E_0\mathbb{I}_\Omega+E_0\mathbb{I}_R, \\
    & u=\mathbb{I}_\Omega+\frac{\Tilde{I}(x)}{\Tilde{b}(x)\sqrt{1+\Tilde{b}'(x)^2}}W_\eta(x)~\mathbb{I}_R, \quad v=\frac{\Tilde{b}'(x)\Tilde{I}(x)}{\Tilde{b}(x)\sqrt{1+\Tilde{b}'(x)^2}}W_\eta(x)~\mathbb{I}_R,
    \end{align}
in which
$$\Tilde{I}(x)=I_\eta(s(x))=\exp(\int_0^x\frac{-\eta\Tilde{b}''(t)}{1+\Tilde{b}'(t)^2}\,\mathrm{d}t),\quad W_\eta(x)=\int_0^x\frac{\Tilde{b}'(t)(-\eta\Tilde{b}'(t)+1)}{\Tilde{I}(t)\sqrt{1+\Tilde{b}'(t)^2}}\,\mathrm{d}t.$$
Moreover, the normal pressure $N$ and skin-friction  $f$ on the ramp are
\begin{align}
    & N(x,\Tilde{b}(x))=\frac{\Tilde{I}(x)\Tilde{b}''(x)}{(1+\Tilde{b}'(x)^2)^\frac{3}{2}}W_\eta(x)+\frac{\Tilde{b}'(x)^2}{1+\Tilde{b}'(x)^2},\label{2.44} \\
    & f(x,\Tilde{b}(x))=\eta N=\eta \Big[\frac{\Tilde{I}(x)\Tilde{b}''(x)}{(1+\Tilde{b}'(x)^2)^\frac{3}{2}}W_\eta(x)+\frac{\Tilde{b}'(x)^2}{1+\Tilde{b}'(x)^2}\Big].\label{eq345}
\end{align}
\end{theorem}
\begin{proof}
    The proof is similar to Theorem 1, so we omit the straightforward calculations here.
\end{proof}

\medskip
We now turn to \textsc{Problem} 3.
Letting $k=0$ in Theorem \ref{sec31} (or $\eta=0$ in Theorem \ref{thm2}), and note that $H_0(x)=W_0(x)=\int_0^x\frac{\Tilde{b}'(t)}{\sqrt{1+\Tilde{b}'(t)^2}}\,\mathrm{d}t$, we obtain the corresponding solution
\begin{align}
    & \varrho=\mathcal{L}^2\mathbb{I}_\Omega+\frac{\Tilde{b}(x)^2}{H_0(x)}~\delta_R,\quad E=E_0\mathbb{I}_\Omega+E_0\mathbb{I}_R, \label{2.46}\\
    & u=\mathbb{I}_\Omega+\frac{H_0(x)}{\Tilde{b}(x)\sqrt{1+\Tilde{b}'(x)^2}}~\mathbb{I}_R,\quad v=\frac{\Tilde{b}'(x)H_0(x)}{\Tilde{b}(x)\sqrt{1+\Tilde{b}'(x)^2}}~\mathbb{I}_R.  \label{2.47}
\end{align}
with the classical Newton-Busemann law \cite[p.71]{new}
\begin{equation}
    N=\frac{\Tilde{b}''(x)}{(1+\Tilde{b}'(x)^2)^\frac{3}{2}}H_0(x)+\frac{\Tilde{b}'(x)^2}{1+\Tilde{b}'(x)^2}. \label{2.48}
\end{equation}
For Problem 3, we assume that the velocity is influenced by $\mu$, for $0<\mu<1$, and the boundary condition \eqref{2.10}  is written explicitly as
\begin{equation}
u(x,\Tilde{b}(x))=\frac{\mu H_0(x)}{\Tilde{b}(x)\sqrt{1+\Tilde{b}'(x)^2}},\quad v(x,\Tilde{b}(x))=\frac{\mu\Tilde{b}'(x)H_0(x)}{\Tilde{b}(x)\sqrt{1+\Tilde{b}'(x)^2}}. \label{2.49}
\end{equation}
We could now specify Problem 3 as follows:

\textsc{Problem $3'$}: Find a solution to \eqref{1.4}-\eqref{1.7}, \eqref{1.9}, \eqref{1.10}($\epsilon=0$), \eqref{1.11} and \eqref{2.49}.

\begin{theorem}\label{thm3}
Problem $3'$ has a Radon measure solution given by
   \begin{align}
    & \varrho=\mathcal{L}^2\mathbb{I}_\Omega+\displaystyle{\frac{\Tilde{b}^2(x)}{\mu  H_0(x)}}\delta_R,\quad E=E_0\mathbb{I}_\Omega+E_0\mathbb{I}_R, \\
    & u=\mathbb{I}_\Omega+\displaystyle{\frac{\mu H_0(x)}{\Tilde{b}(x)\sqrt{1+\Tilde{b}'(x)^2}}}\mathbb{I}_R, \quad v=\displaystyle{\frac{\mu H_0(x)\Tilde{b}'(x)}{\Tilde{b}(x)\sqrt{1+\Tilde{b}'(x)^2}}} \mathbb{I}_R.
\end{align}
   In addition, the normal pressure and skin-friction on the boundary are given by
\begin{align}
    & N(x,\Tilde{b}(x))=\mu H_0(x)\frac{\Tilde{b}''(x)}{(1+\Tilde{b}'(x)^2)^\frac{3}{2}}+\frac{\Tilde{b}'(x)^2}{1+\Tilde{b}'(x)^2},\label{2.52}\\
    & f(x,\Tilde{b}(x))=(1-\mu)\frac{\Tilde{b}'(x)}{1+\Tilde{b}'(x)^2}.\label{eq353}
\end{align}
\end{theorem}
\begin{proof}
    From \eqref{2.25}, \eqref{2.26}, and \eqref{2.49}, we have
    \begin{align*}
        & w_{f_1}(s(x))=\mu\Big(\frac{\Tilde{b}'(x)}{(1+\Tilde{b}'(x)^2)^\frac{3}{2}}-\frac{\Tilde{b}'(x)\Tilde{b}''(x)H_0(x)}{(1+\Tilde{b}'(x)^2)^2}\Big)-\frac{\Tilde{b}'(x)}{\sqrt{1+\Tilde{b}'(x)^2}},\\
        & w_{f_2}(s(x))=\mu \Big(\frac{\Tilde{b}''(x)H_0(x)}{1+\Tilde{b}'(x)^2}-\frac{\Tilde{b}'(x)\Tilde{b}''(x)H_0(x)}{(1+\Tilde{b}'(x)^2)^2}+\frac{\Tilde{b}'(x)^2}{(1+\Tilde{b}'(x)^2)^\frac{3}{2}}\Big),\\
        & w_\rho(s(x))=\frac{\Tilde{b}(x)}{\mu H_0(x)}.
    \end{align*}
    Therefore, the pressure $N$ and the surface friction $f$ on the ramp are
\begin{align*}
       & N=\mu H_0(x)\frac{\Tilde{b}''(x)}{(1+\Tilde{b}'(x)^2)^\frac{3}{2}}+\frac{\Tilde{b}'(x)^2}{1+\Tilde{b}'(x)^2},\quad f=(1-\mu)\frac{\Tilde{b}'(x)}{1+\Tilde{b}'(x)^2}.
\end{align*}
\end{proof}

\begin{remark}{\em
     When $k=0$ (or $\eta=0$), \eqref{2.46}-\eqref{2.48} coincides with Theorem 1.1 in \cite{JQY}. One should notice that even if there is no skin friction acting on the gas, the tangential velocity (momentum) of particles still experience change when they reach the ramp. (A particle does not change its mass when it impinges on the ramp, while its tangential velocity is $(\frac{1}{1+\tilde{b}'(x)^2}, \frac{
     \tilde{b}'(x)}{1+\tilde{b}'(x)^2})$ before it reaches the ramp, which is generally different from $(\frac{H_0(x)}{\Tilde{b}(x)\sqrt{1+\Tilde{b}'(x)^2}}, \frac{\Tilde{b}'(x)H_0(x)}{\Tilde{b}(x)\sqrt{1+\Tilde{b}'(x)^2}})$ given by \eqref{2.47} unless the ramp is straight.) Thus, the assumption of ``what happens to a particle in Newtonian infinite-thin shock layer is the particle striking a surface  loses all of the normal components of its momentum without change in the tangential components" (cf. \cite[p. 173]{HP}) is not rigorous from the mathematical point of view.
}
\end{remark}
\begin{remark}\label{rem36}{\em
Letting $\mu\to0$ in Theorem \ref{thm3}, we have a formal measure solution with $\varrho=\mathcal{L}^2\mathbb{I}_\Omega+\infty\cdot\delta_R$ and
\begin{equation}
u=u_0\mathbb{I}_\Omega=\mathbb{I}_\Omega, ~v=0,~E=E_0\mathbb{I}_\Omega+E_0\mathbb{I}_R.
\end{equation}
At this time the mass concentrated on the ramp is infinite.~Furthermore, the force acting on the boundary is
\begin{equation}
        -w_{f_1}(s(x))=\frac{\Tilde{b}'(x)}{\sqrt{1+\Tilde{b}'(x)^2}},\quad -w_{f_2}(s(x))=0.
\end{equation}
This conclusion is compatible with physical fact since the impulse to the gas should equal to its momentum change  which is along $x$-axis. We can also compute the normal pressure and skin-friction  as
    \begin{align}\label{eq356}
        N(x,\Tilde{b}(x))=\frac{\Tilde{b}'(x)^2}{1+\Tilde{b}'(x)^2},\quad f(x,\Tilde{b}(x))=\frac{\Tilde{b}'(x)}{1+\Tilde{b}'(x)^2}.
    \end{align}
    }
\end{remark}
\begin{remark}{\em
     We call \eqref{2.41}, \eqref{2.44}\eqref{eq345},  \eqref{2.52}\eqref{eq353}, and \eqref{eq356} as the {\em generalized Newton-Busemann law} for the corresponding skin-friction cases. When the ramp is straight with $\Tilde{b}(x)=\tan\theta\cdot x,$ ($\theta\in(0,\pi/2)$ is the tangent inclination angle of the ramp), and there is no skin-frictions, we derive from these formulas that the normal pressure on the ramp is $N=\sin^2\theta$, which is exactly the classical Newton's sine-squared law \cite[p.59]{new}.}
\end{remark}

\section*{Acknowledgments}
This work is supported by the National Natural Science Foundation of China under Grants
No.11871218, No.12071298,  and by Science and Technology Commission of Shanghai Municipality under Grants No.21JC1402500, No.22DZ2229014.

%\begin{thebibliography}{3}

\end{document}